\documentclass{amsart}
\usepackage{graphicx}
\usepackage{amssymb}
\usepackage{amsmath}
\usepackage{amsthm}
\usepackage{subfigure}

\newtheorem{thm}{Theorem}[section]

\newtheorem{lem}[thm]{Lemma}
\newtheorem{prop}[thm]{Proposition}
\theoremstyle{definition}
\newtheorem{defn}[thm]{Definition}
\newtheorem{rem}[thm]{Remark}

\newtheorem*{defn*}{Definition}
\newtheorem*{rems*}{Remarks}
\newtheorem*{rem*}{Remark}

\numberwithin{equation}{section}

\def \local-algebra {\Lambda ^0(\mathbb R^2)/(\nabla H)}

\begin{document}

\title[ Zero-dimensional symplectic ICISs] { Zero-dimensional symplectic isolated complete intersection singularities}
\author{Wojciech Domitrz}
\address{Warsaw University of Technology\\
Faculty of Mathematics and Information Science\\
Plac Politechniki 1\\
00-661 Warsaw\\
Poland\\}

\email{domitrz@mini.pw.edu.pl}
\thanks{The work of W. Domitrz was supported by Polish MNiSW grant no. N N201 397237}

\subjclass{Primary 53D05. Secondary 58K40, 58K50, 58A10, 14H20.}

\keywords{symplectic manifold, local symplectic algebra,
algebraic restrictions, relative Darboux theorem, isolated complete intersection singularities}

\begin{abstract}
We study the local symplectic algebra of the $0$-dimensional
isolated complete intersection singularities. We
use the method of algebraic restrictions to classify these symplectic
singularities. We show that there are non-trivial symplectic invariants
in this classification.
\end{abstract}

\maketitle

\section{Introduction}

The problem of symplectic classification of singular varieties was
introduced by V. I. Arnold in \cite{Ar1}. Arnold showed that the
$A_{2k}$ singularity of a planar curve (the orbit with respect to the
standard $\mathcal A$-equivalence of parameterized curves) split
into exactly $2k+1$ symplectic singularities (orbits with respect
to the symplectic equivalence of parameterized curves). He
distinguished different symplectic singularities by different
orders of tangency of the parameterized curve with the
\emph{nearest} smooth Lagrangian submanifold. Arnold posed a
problem of expressing these new symplectic invariants in terms of
the local algebra's interaction with the symplectic structure and
he proposed to call this interaction the {\bf local symplectic
algebra}. This problem was studied by many authors mainly in the case of singular curves.

In \cite{IJ1} G. Ishikawa and S. Janeczko classified symplectic singularities of curves in the
$2$-dimensional symplectic space.  A symplectic form on a $2$-dimensional manifold
is a special case of a volume form on a smooth manifold. The generalization of results
in \cite{IJ1} to volume-preserving classification of singular varieties and maps  in arbitrary dimensions was obtained in \cite{DR}.
The orbit of the action of all diffeomorphism-germs agrees with the volume-preserving orbit in the $\mathbb C$-analytic category for germs which satisfy a special weak form of quasi-homogeneity e.g. the weak quasi-homogeneity of varieties is a
quasi-homogeneity with non-negative weights $\lambda_i\ge0$ and $\sum_i \lambda_i>0$.

P. A. Kolgushkin classified  stably simple symplectic singularities of parameterized curves
in the $\mathbb C$-analytic category (\cite{K}).

In \cite{DJZ2} the local symplectic algebra of singular
quasi-homogeneous  subsets of a symplectic space was explained by
the algebraic restrictions of the symplectic form to these
subsets.  The generalization of the Darboux-Givental theorem (\cite{ArGi})
to germs of arbitrary subsets of the symplectic space obtained in \cite{DJZ2} reduces
the problem of symplectic classification of germs of quasi-homo\-ge\-neous subsets to
the problem of classification of algebraic restrictions of symplectic
forms to these subsets. For non-quasi-homogeneous subsets there is one more cohomological invariant apart of the algebraic restriction (\cite{DJZ2},
\cite{DJZ1}). The method of algebraic restrictions is a very powerful tool to study the local symplectic algebra of $1$-dimensional singular analytic varieties  since the space of algebraic restrictions of closed
$2$-forms to a $1$-dimensional singular analytic variety is finite-dimensional (\cite{D}). By this method  complete symplectic classifications of the $A-D-E$ singularities of planar curves and the $S_5$ singularity were obtained in \cite{DJZ2}. These results were generalized to other $1$-dimensional isolated complete intersection singularities: the $S_{\mu}$ symplectic singularities for $\mu>5$ in \cite{DT1}, the $T_7-T_8$ symplectic singularities in \cite{DT2} and the $W_8-W_9$ symplectic singularities in \cite{T}.

In this paper we show that some non-trivial symplectic invariants appear not only in the case of singular curves but also in the case of multiple points.  We consider the symplectic classification of the
$0$-dimensional isolated complete intersection singularities (ICISs) in the symplectic space $(\mathbb C^{2n},\omega)$. We need to introduce a symplectic $V$-equivalence to study this problem since we consider the ideals of function-germs that have not got the property of zeros.

We
recall that $\omega$ is a $\mathbb C$-analytic symplectic form on $\mathbb C^{2n}$ if $\omega$ is a $\mathbb C$-analytic
nondegenerate closed 2-form, and $\Phi:\mathbb{C}^{2n}
\rightarrow\mathbb{C}^{2n}$ is a symplectomorphism if $\Phi$ is a
$\mathbb C$-analytic diffeomorphism and $\Phi ^* \omega=\omega$.

\begin{defn} \label{symplecto}
Let $f, g :(\mathbb C^{2n},0)\rightarrow (\mathbb C^{k},0)$ be $\mathbb C$-analytic map-germs on the symplectic space $(\mathbb{C}^{2n}, \omega)$. $f, g$
are \textbf{symplectically $V$-equivalent} if there exist a symplecto\-morphism-germ $\Phi:(\mathbb{C}^{2n}, 0,\omega) \rightarrow(\mathbb{C}^{2n}, 0, \omega)$ and a $\mathbb C$-analytic map-germ $M: (\mathbb C^{2n},0)\rightarrow GL(k,\mathbb C)$ such that
such that $f \circ \Phi=M \cdot g$.
\end{defn}

If $\Phi:(\mathbb C^n,0)\rightarrow (\mathbb C^m,0)$ is a $\mathbb C$-analytic map-germ then for an ideal $I$ in the ring of $\mathbb C$-analytic function-germs on $\mathbb C^m$ we denote by $\Phi^{\ast}I$ the following ideal $\{f\circ\Phi:f\in I\}$ in the ring of $\mathbb C$-analytic function-germs on $\mathbb C^n$. The (symplectic) $V$-equivalence of map-germs $f=(f_1,\cdots,f_k), g=(g_1,\cdots,g_k):(\mathbb C^{2n},0)\rightarrow (\mathbb C^{k},0)$ corresponds to the following (symplectic) equivalence of finitely-generated ideals $<f_1,\cdots,f_k>$ and $<g_1,\cdots, g_k>$ (see \cite{AVG}).

\begin{defn}
Ideals $<f_1,\cdots,f_k>$ and $<g_1,\cdots, g_k>$ of $\mathbb C$-analytic function-germs at $0$ on the symplectic space $(\mathbb C^{2n},\omega)$ are {\bf symplectically equivalent} if there exists a symplecto\-morphism-germ $\Phi:(\mathbb{C}^{2n}, 0,\omega) \rightarrow(\mathbb{C}^{2n}, 0, \omega)$ such that $\Phi^{\ast}<f_1,\cdots,f_k>=<g_1,\cdots,g_k>$.
\end{defn}

In this paper we present the complete symplectic classification of  the $I_{a,b}$, $I_{2a+1}$, $I_{2a+4}$, $I_{a+5}$, $I^{\ast}_{10}$ singularities. For $n=1$ all $V$-orbits coincide with symplectic $V$-orbits. The situation for $n\ge 2$ is different: the $I_{a,b}$ singularities split into two symplectic $V$-orbits, the $I_{2a+1}$, $I_{2a+4}$, $I_{a+5}$ singularities split into three symplectic orbits and finally $I^{\ast}_{10}$ singularity splits into four symplectic $V$-orbits. The symplectic $V$-orbits of a map $f=(f_1,\cdots,f_{2n})$ are distinguished by the order of vanishing of a pullback of the germ of the symplectic form to a $\mathbb C$-analytic non-singular submanifold $M$ of the minimal dimension such that the ideal of $\mathbb C$-analytic function-germs vanishing $M$ is contained in the ideal $<f_1,\cdots,f_{2n}>$ (see Definition \ref{index}).

To obtain these results we need some reformulation and modification of the method of algebraic restrictions. We  present it in Section \ref{method}.
In Section \ref{discrete} we give the definitions of discrete symplectic invariants which completely distinguish symplectic $V$-singularities considered in this paper. We recall basic facts on the classification of $V$-simple maps in Section \ref{V-simple}. In Section \ref{symplICIS} we prove the symplectic $V$-classification theorem for $0$-dimensional ICISs (Theorem \ref{s-main}).

\section{The method of algebraic restrictions for the symplectic $V$-equivalence.}
\label{method}

In this section we present basic facts on the method of algebraic
restrictions adapted to the case of the symplectic $V$-equivalence.
The proofs of all results are small modifications of the proofs of analogous results in \cite{DJZ2}.

Given a germ at $0$ of a non-singular $\mathbb C$-analytic submanifold $M$ of $\mathbb C^m$ denote by $\Lambda ^p(M)$ the space of all  germs at $0$ of $\mathbb C$-analytic
differential $p$-forms on $M$. By $\mathcal O(M)$ denote the ring of $\mathbb C$-analytic function-germs on $M$ at $0$. Given an ideal  $I$ in $\mathcal O(M)$
introduce the following subspace of $\Lambda ^p(M)$:
$$\mathcal A^p_0(I, M) = \{\alpha  + d\beta : \ \ \alpha
\in I\Lambda^p(M), \ \beta \in I\Lambda^{p-1}(M).\}$$ The
relation $\omega\in I\Lambda^p(M)$ means that $\omega=\sum_{i=1}^k f_i \alpha_i$,
where $\alpha_i\in \Lambda^p(M)$ and $f_i\in I$ for $i=1,...,k$.

\smallskip

\begin{defn}
\label{main-def} Let $I$ be an ideal of $\mathcal O(M)$ and let
$\omega \in \Lambda ^p(M)$. The {\bf algebraic restriction} of
$\omega $ to $I$ is the equivalence class of $\omega $ in $\Lambda
^p(M)$, where the equivalence is as follows: $\omega $ is
equivalent to $\widetilde \omega $ if $\omega - \widetilde \omega
\in \mathcal A^p_0(I,M)$.
\end{defn}

\noindent {\bf Notation}. The algebraic restriction of the germ of
a $p$-form $\omega $ on $M$ to the ideal $I$ in $\mathcal O(M)$
will be denoted by $[\omega ]_I$. Writing $[\omega ]_I=0$ (or
saying that $\omega $ has zero algebraic restriction to $I$) we
mean that $[\omega ]_I = [0]_I$, i.e. $\omega \in A^p_0(I,M)$.

\begin{defn}Two algebraic restrictions
$[\omega ]_I$ and $[\widetilde \omega ]_{\widetilde I}$ are called {\bf
diffeomorphic} if there exists the germ of a diffeomorphism $\Phi:
 M\rightarrow \widetilde M$ such that $\Phi^{\ast}(\widetilde I)=I$ and  $[\Phi^{\ast}\widetilde \omega ]_I=[\omega ]_I$.
\end{defn}

\begin{defn}
The germ of a function, a differential $k$-form, or a vector field
$\alpha$ on $(\mathbb C^m,0)$ is {\bf quasi-homogeneous} in a
coordinate system $(x_1,\cdots,x_m)$ on $(\mathbb C^m,0)$ with
positive integer weights $(\lambda_1,\cdots, \lambda_m)$ if $\mathcal L_E
\alpha=\delta \alpha$, where $E=\sum_{i=1}^m\lambda_i
x_i\frac{\partial}{\partial x_i}$ is the germ of the {\bf Euler vector
field} on $(\mathbb C^m,0)$ and the integer $\delta$ is called
the quasi-degree.
\end{defn}

 It is easy to show that $\alpha$ is quasi-homogeneous in a coordinate system
 $(x_1,\cdots,x_m)$ with weights $(\lambda_1,\cdots, \lambda_m)$ if and only if
 $F_t^{\ast}\alpha=t^{\delta}\alpha$, where
 \begin{equation}\label{Ft}
 F_t(x_1,\cdots,x_m)=(t^{\lambda_1} x_1,\cdots, t^{\lambda_m}x_m).
 \end{equation}

\begin{defn}
A finitely generated ideal $I$ of $\mathcal O(\mathbb C^m)$ is {\bf quasi-homogeneous} if there exist generators of $I$ which are quasi-homogeneous in  the same coordinate system $(x_1,\cdots,x_m)$ on $\mathbb C^m$ with the same positive integer weights $(\lambda_1,\cdots, \lambda_m)$.

A map-germ $f=(f_1,\cdots,f_k):(\mathbb C^{m},0)\rightarrow (\mathbb C^{k},0)$ is {\bf quasi-homogeneous} if function-germs $f_1,\cdots,f_k$ are quasi-homogeneous in  the same coordinate system $(x_1,\cdots,x_m)$ on $\mathbb C^m$ with the same positive integer weights $(\lambda_1,\cdots, \lambda_m)$.
\end{defn}

To prove the generalization of Darboux-Givental theorem suitable for the symplectic $V$-equivalence of maps or the symplectic equivalence of ideals of function-germs we need the following version of the Relative Poincar\'{e} Lemma.

\begin{lem}\label{Poincare}
Let $I$ be a finitely generated quasi-homogeneous ideal in $\mathcal O(\mathbb C^m)$.
If $\omega\in I\Lambda^p(\mathbb C^m)$ is closed than there exists $\alpha \in I\Lambda^{p-1}(\mathbb C^m)$ such that $\omega=d\alpha$.
\end{lem}

\begin{proof} We use the method described in \cite{DJZ1}. We can find a coordinate system $(x_1,\cdots,x_m)$ on $(\mathbb C^m,0)$ and positive integer weights $(\lambda_1,\cdots,\lambda_m)$ and quasi-homogeneous function-germs $f_1,\cdots,f_k\in \mathcal O(\mathbb C^m)$ (in this coordinate systems with these weights) such that $I=<f_1,\cdots,f_k>$. Let $\delta_i$ be a quasi-degree of $f_i$ for $i=1,\cdots,k$.

Let $F_t$ be a map defined in (\ref{Ft}) and let $V_t$ be a
vector field along $F_t$ for $t\in[0;1]$ such that $V_t\circ F_t=F_t^{\prime}$.

Then we have  $F_0^{\ast}\omega =0$ and it implies that
$$\omega  = F_1^{\ast}\omega -F_0^{\ast}\omega  =\int_0^1(F_t^{\ast}
\omega )^\prime
dt=\int_0^1F_t^{\ast}d(V_t\rfloor \omega )dt=d\left(\int_0^1F_t^{\ast}(V_t\rfloor \omega )dt\right).$$
Let $\alpha=\int_0^1 F_t^{\ast}(V_t\rfloor \omega )dt$, then $\omega=d\alpha$. But $\omega$ belongs to $I\Lambda^p(\mathbb C^m)$. It implies that there exist germs of $p$-forms $\beta_i$ in $\Lambda^p(\mathbb C^m)$ for $i=1,\cdots,k$ such that $\omega=\sum_{i=1}^k f_i\beta_i$.
So we have that
$$
\alpha=\int_0^1 F_t^{\ast}(V_t\rfloor \sum_{i=1}^k f_i\beta_i )dt=\sum_{i=1}^k f_i \int_0^1 t^{\delta_i} F_t^{\ast}(V_t\rfloor \beta_i)dt.
$$
Thus $\alpha$ belongs to $I\Lambda^{p-1}(\mathbb C^m)$.
\end{proof}

The method of algebraic restrictions applied to finitely-generated
quasi-homogeneous ideals is based on the following theorem.

\begin{thm}[a modification of Theorem A in \cite{DJZ2}] \label{thm A}
Let $I$ be a finitely gene\-rated quasi-homoge\-neous ideal in  $\mathcal O(\mathbb
C^{2n})$.
\begin{enumerate}
\item
If $\omega _0, \omega _1$ are germs at $0$ of symplectic forms
on $\mathbb C^{2n}$ with the same algebraic restriction to $I$ then
there exists a $\mathbb C$-analytic diffeomorphism-germ $\Phi $  of $\mathbb C^{2n}$ at $0$ of the form $\Phi (x)=(x_1+\phi_1(x),\cdots,x_{2n}+\phi_{2n}(x))$, where $\phi_i\in I$ for $i=1,\cdots,2n$, such that $\Phi ^*\omega _1 = \omega _0$.
\item
$\mathbb C$-analytic quasi-homogeneous map-germs $f=(f_1,\cdots,f_k), g=(g_1,\cdots,g_k) :(\mathbb C^{2n},0)\rightarrow (\mathbb C^{k},0)$ on the symplectic space $(\mathbb{C}^{2n}, \omega)$ are symplectically $V$-equivalent if and only if algebraic restrictions $[\omega]_{<f_1,\cdots,f_k>}$ and $[\omega]_{<g_1,\cdots,g_k>}$ are diffeomorphic.
\end{enumerate}
\end{thm}

\begin{rem}
It is obvious that if $\Phi (x)=(x_1+\phi_1(x),\cdots,x_{2n}+\phi_{2n}(x))$ where $\phi_i\in I$ for $i=1,\cdots,2n$ then $\Phi ^*I = I$
\end{rem}

A proof of Theorem \ref{thm A} can be obtain by a small modification of the proof of Theorem A in \cite{DJZ2}. One only needs Lemma \ref{Poincare} and the following fact.

\begin{lem}
Let $I$ be a finitely generated ideal in $\mathcal O(\mathbb C^m)$. Let $X_t=\sum_{i=1}^m f_{i,t} \frac{\partial}{\partial x_i}$ for $t\in [0;1]$ be a family of germs of $\mathbb C$-analytic vector fields on $\mathbb C^m$ such that $f_{i,t}\in I$ for $i=1,\cdots,m$.

If $\Phi_t$ for $t\in [0,1]$ is a family of diffeomorphism-germs of $(\mathbb C^m,0)$ such that
\begin{equation}\label{ODE}
\frac{d}{dt}\Phi_t=X_t\circ \Phi_t
\end{equation}
 then
\begin{equation}\label{inIdeal}
\Phi_t (x)=(x_1+\phi_{1,t}(x),\cdots,x_{2n}+\phi_{2n,t}(x)),
\end{equation}
where $\phi_{i,t}\in I$ for $i=1,\cdots,2n$.
\end{lem}

\begin{proof}[A sketch of the proof]
The map $t\mapsto \Phi_t(x)$ is a solution of ODE $\frac{dy}{dt}=X_t(y)$ with the initial condition $y(0)=x$. So $\Phi_t(x)$ can be obtained as a limit $\lim_{n\to \infty}T^n\Psi$ where $\Psi(t,x)\equiv x$ and $(T\Psi)(t,x)=x+\int_0^tX_s(\Psi(s,x))ds$ is the Picard's operator.  It is easy to see that if $\Psi$ has the form  (\ref{inIdeal}) then $T\Psi$ has the form (\ref{inIdeal}) too. The ideal $I$ is finitely generated. Thus $\Phi_t$ has also this form.
\end{proof}

Theorem \ref{thm A} reduces the problem of symplectic
classification of  quasi-homogeneous ideals to
the problem of  classification of the algebraic
restrictions of the germ of the symplectic form to
quasi-homogeneous ideals.

The meaning of the zero algebraic restriction is explained
by the following theorem.

\begin{thm}[a modification of Theorem {\bf B} in \cite{DJZ2}] \label{thm B}  A finitely generated quasi-homogeneous
 ideal $I$  of $\mathcal O(\mathbb C^{2n})$ contains the ideal of $\mathbb C$-analytic function-germs vanishing on the germ of a non-singular Lagrangian submanifold of the symplectic space
$(\mathbb C^{2n}, \omega )$  if and only if the symplectic form $\omega$ has zero algebraic restriction to $I$.
\end{thm}

We now formulate the modifications of  basic properties of
algebraic restrictions (\cite{DJZ2}). First we can reduce the dimension of the manifold  due
to the following propositions.

If the ideal $I$ in $\mathcal O(\mathbb C^m)$  contains an ideal $I(M)$ of function-germs vanishing on a
non-singular submanifold $M\subset \mathbb C^m$ then the
classification of the algebraic restrictions to $I$ of $p$-forms on
$\mathbb C^m$ reduces to the classification of the algebraic
restrictions to $I|_M=\{f|_M:f\in I\}$ of $p$-forms on $M$. At first note that the
algebraic restrictions $[\omega ]_I$ and $[\omega \vert
_{TM}]_{I|_M}$ can be identified:

\begin{prop}
\label{reduction} Let $I$ be an ideal in $\mathcal O(\mathbb C^m)$ which contains an ideal of function-germs vanishing on a
non-singular submanifold $M\subset \mathbb C^m$ and let $\omega _1, \omega _2$ be germs of $p$-forms on
$\mathbb C^m$. Then $[\omega _1]_I = [\omega _2]_I$ if and only if
$\bigl[\omega _1\vert _{TM}\bigr]_{I|_M} = \bigl[\omega _2 \vert
_{TM}\bigr]_{I|_M}$.
\end{prop}

The following, less obvious statement, means that the {\it orbits}
 of the algebraic restrictions $[\omega ]_I$ and $[\omega \vert
_{TM}]_{I|_M}$ also can be identified.

\begin{prop}
\label{main-reduction} Let $I_1,I_2$ be ideals in the ring
$\mathcal O(\mathbb C^m)$, which contain $I(M_1)$ and $I(M_2)$ respectively, where $M_1, M_2$ are  equal-dimensional non-singular
submanifolds. Let $\omega _1, \omega _2$
be two germs of $p$-forms. The algebraic restrictions $[\omega
_1]_{I_1}$ and $[\omega _2]_{I_2}$ are diffeomorphic if and only
if the algebraic restrictions $\bigl[\omega _1\vert
_{TM_1}\bigr]_{I_1|_{M_1}}$ and $\bigl[\omega _2\vert
_{TM_2}\bigr]_{I_2|_{M_2}}$ are diffeomorphic.
\end{prop}

To calculate the space of algebraic restrictions of germs of $2$-forms we
will use the following obvious properties.

\begin{prop}\label{d-wedge}
If $\omega\in \mathcal A_0^k(I,\mathbb C^{2n})$ then $d\omega
\in \mathcal A_0^{k+1}(I,\mathbb C^{2n})$ and $\omega\wedge
\alpha \in \mathcal A_0^{k+p}(I,\mathbb C^{2n})$ for any germ of $\mathbb C$-analytic
$p$-form $\alpha$ on $\mathbb C^{2n}$.
\end{prop}

Then we need to determine which algebraic restrictions of closed
$2$-forms are realizable by symplectic forms. This is possible due
to the following fact.

\begin{prop}\label{rank}
 Let $I$ be an ideal of $\mathcal O(\mathbb C^{2n})$. Let $r$ be
the minimal dimension of non-singular submanifolds $M$ of $\mathbb
C^{2n}$ such that $I$ contains the ideal $I(M)$.  The algebraic restriction $[\theta ]_I$ of the germ
of a closed $2$-form $\theta $ is realizable by the germ of a
symplectic form on $\mathbb C^{2n}$ if and only if $rank (\theta
\vert _{T_0M})\ge 2r - 2n$.
\end{prop}

\section{Discrete symplectic invariants.}\label{discrete}

We use discrete symplectic invariants to distinguish symplectic singularity classes. We modify definitions of these invariants introduced in \cite{DJZ2} for the symplectic $V$-equivalence.

 The first invariant is a symplectic
multiplicity (\cite{DJZ2}) introduced  in \cite{IJ1} as a
symplectic defect of a curve.

Let $f:(\mathbb C^{2n},0)\rightarrow (\mathbb C^{k},0)$ be the germ of a $\mathbb C$-analytic map on the symplectic space $(\mathbb C^{2n},\omega)$.
\begin{defn}
\label{def-mu}
 The {\bf symplectic multiplicity} $\mu_{sympl}(f)$ of  $f$ is the codimension of
 the symplectic $V$-orbit of $f$ in the $V$-orbit of $f$.
\end{defn}

The second invariant is the index of isotropy \cite{DJZ2}.

\begin{defn} \label{index}
The {\bf index of isotropy} $\iota(f)$ of $f=(f_1,\cdots,f_k)$ is the maximal
order of vanishing of the $2$-forms $\omega \vert _{TM}$ over all
smooth submanifolds $M$ such that the ideal $<f_1,\cdots,f_k>$ contains $I(M)$.
\end{defn}

These invariants can be described in terms of algebraic restrictions.

\begin{prop}[\cite{DJZ2}]\label{sm}
The symplectic multiplicity  of the germ of a quasi-homogeneous map $f=(f_1,\cdots,f_k)$ on the
symplectic space $(\mathbb C^{2n},\omega)$ is equal to the codimension of the orbit of the
algebraic restriction $[\omega ]_{<f_1,\cdots,f_k>}$ with respect to the group of
diffeomorphism-germs preserving the ideal $<f_1,\cdots,f_k>$  in the space of the algebraic
restrictions of closed  $2$-forms to $<f_1,\cdots,f_k>$.
\end{prop}

\begin{prop}[\cite{DJZ2}]\label{ii}
The index of isotropy  of the germ of a quasi-homogeneous map $f=(f_1,\cdots,f_k)$ on the
symplectic space $(\mathbb C^{2n}, \omega )$ is equal to the
maximal order of vanishing of closed $2$-forms representing the
algebraic restriction $[\omega ]_{<f_1,\cdots,f_k>}$.
\end{prop}

We will use these invariants to distinguish symplectic singularities.
\section{$V$-simple maps}\label{V-simple}

We recall some results on classification of $V$-simple germs (for details see
\cite{AVG}).
\begin{defn}
The germ $f:(\mathbb C^m,0)\rightarrow (\mathbb C^n,0)$ is said be
{\bf $V$-simple} if its $k$-jet, for any $k$, has a neighborhood
in the small jet space $J^k_{0,0}(\mathbb C^m, \mathbb C^n)$ that
intersects only a finite number of $V$-equivalence classes
(bounded by a constant independent of $k$).
\end{defn}

\begin{defn}
The {\bf $p$-parameter suspension of the map-germ} $f:(\mathbb
C^m,0)\rightarrow (\mathbb C^n,0)$ is the map germ
$$
F:(\mathbb C^m\times \mathbb C^p,0)\ni (y,z)\mapsto (f(y),z)\in
(\mathbb C^n\times \mathbb C^p,0).
$$
\end{defn}

\begin{thm}[see \cite{AVG}]
The V-simple map-germs $(\mathbb C^m,0)\rightarrow (\mathbb C^n,0)$
with $m\ge n$ belong, up to $V$-equivalence and suspension, to one
of the three lists: the $A-D-E$ singularities of map-germs $\mathbb C^m\rightarrow \mathbb C$ (hypersurfaces with an isolated singularity), $S-T-U-W-Z$ singularities of map-germs $\mathbb C^3\rightarrow \mathbb C^2$ ($1$-dimensional ICISs) and singularities of map-germs $\mathbb C^2\rightarrow \mathbb C^2$ ($0$-dimensional ICISs) presented in Table \ref{I}.
\end{thm}

\begin{center}
\begin{table}[h]

    \begin{small}
    \noindent
    \begin{tabular}{|c|c|c|}

            \hline
    Notation &  Normal form & Restrictions  \\ \hline

     $I_{a,b}$ & $(yz, y^a+z^b$) & $a\ge b\ge 2$ \\ \hline

$I_{2a+1}$ & $(y^2+z^3, z^a$) & $a\ge 3$ \\ \hline

$I_{2a+4}$ &$(y^2+z^3, yz^a)$ & $a\ge 2$ \\ \hline

$I_{a+5}$ & $(y^2+z^a, yz^2)$ & $a\ge 4$ \\ \hline

$I^*_{10}$ & $(y^2, z^4$) & - \\ \hline
\end{tabular}

\smallskip

\caption{\small V-simple map-germs $\mathbb C^2\rightarrow \mathbb
C^2$. }\label{I}

\end{small}
\end{table}
\end{center}
The normal forms in Table \ref{I} were obtained  in \cite{G} by M. Giusti.
\section{Symplectic $0$-dimensional ICISs}\label{symplICIS}

We use the method of algebraic restrictions to obtain a complete classification of
singularities presented in Table \ref{I}.

\begin{thm}\label{s-main}
Any map-germ $(\mathbb C^{2n},0)\rightarrow (\mathbb C^{2n},0)$ from the symplectic space
$(\mathbb
C^{2n},\sum_{i=1}^n dp_i\wedge dq_i)$ which is $V$-equivalent (up to a suitable suspension)
to one of the normal forms in Table \ref{I} is symplectically $V$-equivalent to one and only one of the following
normal forms presented in Table \ref{sI}

\begin{center}
\begin{table}[h]

    \begin{small}
    \noindent
    \begin{tabular}{|c|c|c|c|c|}

            \hline
    Symplectic class &   Normal forms    & cod & $\mu_{sympl}$ &  i  \\ \hline
  $I_{a,b}^0$, $(n\ge 1)$ & $(p_1q_1,p_1^a+q_1^b,p_2,q_2,\cdots,p_n,q_n)$  &  $0$ & $0$ & $0$  \\  \hline

  $I_{a,b}^1$, $(n\ge 2)$  & $(p_1p_2,p_1^a+p_2^b,q_1,q_2,p_3,q_3,\cdots,p_n,q_n)$ &  $1$ & $1$ & $\infty$ \\ \hline
  \hline
$I_{2a+1}^0$, $(n\ge 1)$& $(p_1^2+q_1^3,q_1^a,p_2,q_2,\cdots,p_n,q_n)$
              & $0$ & $0$ & $0$     \\ \hline

   $I_{2a+1}^1$, $(n\ge 2)$& $(p_1^2+p_2^3,p_2^a,q_1,q_2+p_1p_2,p_3,q_3,\cdots,p_n,q_n)$
                                         &$1$  & $1$ & $1$ \\ \hline

    $I_{2a+1}^2$,  $(n\ge 2)$ & $(p_1^2+p_2^3,p_2^a,q_1,q_2,p_3,q_3,\cdots,p_n,q_n)$
                                         &$2$  & $2$ & $\infty$ \\ \hline \hline
$I_{2a+4}^0$, $(n\ge 1)$& $(p_1^2+q_1^3,p_1q_1^a,p_2,q_2,\cdots,p_n,q_n)$
              & $0$ & $0$ & $0$     \\ \hline
$I_{2a+4}^1$, $(n\ge 2)$& $(p_1^2+p_2^3,p_1p_2^a,q_1,q_2+p_1p_2,p_3,q_3,\cdots,p_n,q_n)$
              & $1$ & $1$ & $1$     \\ \hline
$I_{2a+4}^2$, $(n\ge 2)$& $(p_1^2+p_2^3,p_1p_2^a,q_1,q_2,p_3,q_3,\cdots,p_n,q_n)$
              & $2$ & $2$ & $\infty$     \\ \hline  \hline
$I_{a+5}^0$, $(n\ge 1)$& $(p_1^2+q_1^a,p_1q_1^2,p_2,q_2,\cdots,p_n,q_n)$
              & $0$ & $0$ & $0$     \\ \hline
$I_{a+5}^1$, $(n\ge 2)$& $(p_1^2+p_2^a,p_1p_2^2,q_1,q_2+p_1p_2,p_3,q_3,\cdots,p_n,q_n)$
              & $1$ & $1$ & $1$     \\ \hline
$I_{a+5}^1$, $(n\ge 2)$& $(p_1^2+p_2^a,p_1p_2^2,q_1,q_2,p_3,q_3,\cdots,p_n,q_n)$
              & $2$ & $2$ & $\infty$     \\ \hline   \hline
$I_{10}^{\ast 0}$, $(n\ge 1)$& $(p_1^2,q_1^4,p_2,q_2,\cdots,p_n,q_n)$
              & $0$ & $0$ & $0$     \\ \hline
$I_{10}^{\ast 1}$, $(n\ge 2)$& $(p_1^2,p_2^4,q_1,q_2+p_1p_2,p_3,q_3,\cdots,p_n,q_n)$
              & $1$ & $1$ & $1$     \\ \hline
$I_{10}^{\ast 2}$, $(n\ge 2)$& $(p_1^2,p_2^4,q_1,q_2+p_1p_2^2,p_3,q_3,\cdots,p_n,q_n)$
              & $2$ & $2$ & $2$     \\ \hline
$I_{10}^{\ast 3}$, $(n\ge 2)$& $(p_1^2,p_2^4,q_1,q_2,p_3,q_3,\cdots,p_n,q_n)$
& $3$ & $3$ & $\infty$     \\ \hline
\end{tabular}

\caption{\small Classification of symplectic $0$-dimensional isolated complete intersection singularities,
$cod$ -- codimension of the classes; \ $\mu_{sympl}$-- symplectic multiplicity; \ $i$ --  index of isotropy.}\label{sI}

\end{small}
\end{table}
\end{center}

\end{thm}

\begin{proof} In the case $n=1$ the proof follows from results in \cite{DR} where it was proved that for quasi-homogeneous singularities in the $\mathbb C$-analytic category  $V$-orbits coincide with volume-preserving $V$-orbits. For general $n$ we present the proof in the case of the $I_{10}^{\ast}$ singularity where there are $4$ different symplectic singularity classes, and in the case of the $I_{a+5}$ singularity. The proofs in other cases are very similar.

For the $I_{10}^{\ast}$ singularity  we calculate the space of algebraic restrictions of $2$-forms to the ideal $I=<y^2,z^4,x_1,\cdots,x_{2n-2}>$. The ideal generated by $x_1,\cdots,x_{2n-2}$ is contained in $I$. So by Proposition \ref{reduction} we may consider the following ideal $J=I|_{\{x_1=\cdots=x_{2n-2}=0\}}=<y^2,z^4>$ in the ring $\mathcal O(\mathbb C^2)$. By Proposition \ref{d-wedge}  germs of $1$-forms $d(1/2y^2)=ydy, d(1/4z^4)=z^3dz$   and germs of $2$-forms $ydy\wedge dz, z^3dy\wedge dz$ have zero algebraic restriction to $J$. So any algebraic restriction of the germ of a closed $2$-forms to $J$ can be  presented in  the following form $[\omega]_J=A[dy\wedge dz]_J +B[zdy \wedge dz]_J+C[z^2dy\wedge dz]_J$, where $A,B,C\in \mathbb C$.

If $A\ne 0$ then we obtain $\Phi^{\ast}[\omega]_J=[dy\wedge dz]_J$ by the diffeomorphism-germ of the form $\Phi(y,z)=(y,z(A+1/2Bz+1/3Cz^2))$.
If $A=0$ and $B\ne 0$ then we obtain $\Phi^{\ast}[\omega]_J=[zdy\wedge dz]_J$ by the diffeomorphism-germ of the form $\Phi(y,z)=(y,z\phi(z))$, where $\phi^2(z)=B+2/3Cz$.
If $A=B=0$ and $C\ne0$ then we obtain $\Phi^{\ast}[\omega]_J=[z^2dy\wedge dz]_J$ by the diffeomorphism-germ of the form $\Phi(y,z)=(Cy,z)$.

Since the minimal dimension $r$ of the germ of a non-singular submanifold $M$ such that $I(M)\subset I$ is $2$ then by Proposition \ref{rank} for $n=1$ only the algebraic restriction $[dy\wedge dz]_I$ is realizable by the germ of a symplectic form.

For $n>1$ all algebraic restrictions are realizable by the following symplectic forms:
\begin{equation}\label{non-zero}
dy \wedge dz + \sum_{i=1}^{n-1}dx_{2i-1}\wedge dx_{2i},
\end{equation}
\begin{equation}\label{z}
zdy \wedge dz + dy\wedge dx_1+dz\wedge dx_2+\sum_{i=2}^{n-1}dx_{2i-1}\wedge dx_{2i},
\end{equation}
\begin{equation}\label{z2}
z^2dy \wedge dz + dy\wedge dx_1+dz\wedge dx_2+\sum_{i=2}^{n-1}dx_{2i-1}\wedge dx_{2i},
\end{equation}
\begin{equation}\label{zero}
dy\wedge dx_1+dz\wedge dx_2+\sum_{i=2}^{n-1}dx_{2i-1}\wedge dx_{2i}.
\end{equation}
By a simple change of coordinates we obtain the normal forms in Table \ref{sI}.

For the $I_{a+5}$ singularity the space algebraic restrictions of germs of closed $2$-forms to the ideal $I=<y^2+z^a,yz^2,x_1,\cdots,x_{2n-2}>$ can calculated in the same way. We obtain that any algebraic restriction of the germs of a closed $2$-forms on $\mathbb C^2=\{x_1=\cdots=x_{2n-2}=0\}$ to $J=I|_{\{x_1=\cdots=x_{2n-2}=0\}}=<y^2+z^a,yz^2>$ can be  presented in  the following form
\begin{equation}\label{al-r}
[\omega]_J=A[dy\wedge dz]_J +B[zdy \wedge dz]_J,
\end{equation}
where $A,B\in \mathbb C$.

First assume that $A\ne 0$.
Let $E$ denote the germ of the Euler vector field $ay\frac{\partial}{\partial y}+2z\frac{\partial}{\partial y}$. Then it is easy to check that a flow $\Phi_t$ of the germ of a vector field $X=\frac{B}{(a+4)A}zE$ preserves $J$, $\mathcal L_{X}(Ady\wedge dz)=Bz dy\wedge dz$, $[\mathcal L_{X}(Bzdy\wedge dz)]_J=0$. Therefore $\Phi_t^{\ast}[Ady\wedge dz + t Bz dy\wedge dz]_J=[Ady\wedge dz]_J$ for $t\in [0;1]$ (see \cite{D}).
Finally by a linear change of coordinates of the form $(y,z)\mapsto (Cy,Dz)$, where for $C,D\in \mathbb C$ such that $C^2=D^a$ and $CD=A$ we show that if $A\ne 0$ then the algebraic restriction (\ref{al-r}) is diffeomorphic to $[dy\wedge dz]_J$. By a similar  change of coordinates preserving $J$ we show that if $A=0$ and $B\ne 0$ then the algebraic restriction (\ref{al-r}) is diffeomorphic to $[zdy\wedge dz]_J$. As in the previous case, for $n=1$ only $[dy\wedge dz]_I$ can be realizable by the germ of a symplectic form . For $n\ge 2$ algebraic restrictions are realizable by (\ref{non-zero}), (\ref{z}) and (\ref{zero}). Normal forms in Table \ref{sI} are obtained by an obvious change of coordinates.
\end{proof}

\bibliographystyle{amsalpha}

\begin{thebibliography}{AAAA}

\bibitem [A1] {Ar1} V. I. Arnold, \emph{ First step of local symplectic algebra},
Differential topology, infinite-dimensional Lie algebras, and
applications. D. B. Fuchs' 60th anniversary collection.
Providence, RI: American Mathematical Society. Transl., Ser. 2,
Am. Math. Soc. 194(44), 1999,1-8.

\bibitem [AG] {ArGi} V. I. Arnold, A. B. Givental \emph{Symplectic geometry}, in Dynamical systems, IV, 1-138, Encyclopedia of Matematical Sciences, vol. 4,
Springer, Berlin, 2001.

\bibitem [AVG] {AVG} V. I. Arnold, S. M. Gusein-Zade, A. N. Varchenko,
\emph{ Singularities of Differentiable Maps}, Vol. 1, Birhauser,
Boston, 1985.


\bibitem [D] {D} W. Domitrz,
\emph{Local symplectic algebra of quasi-homogeneous curves}, Fundamentae Mathematicae 204 (2009), 57-86.

\bibitem [DJZ1] {DJZ1} W. Domitrz, S. Janeczko, M. Zhitomirskii,
\emph{Relative Poincar\'e lemma, contractibility, quasi-homogeneity
and vector fields tangent to a singular variety}, Ill. J. Math.
48, No.3 (2004), 803-835.

\bibitem [DJZ2] {DJZ2} W. Domitrz, S. Janeczko, M. Zhitomirskii,
\emph{Symplectic singularities of varietes: the method of
algebraic restrictions}, J. reine und angewandte Math. 618 (2008),
197-235.

\bibitem[DR] {DR} W. Domitrz, J. H. Rieger, \emph{Volume preserving subgroups of
$\mathcal A$ and $\mathcal K$ and singularities in unimodular
geometry}, Mathematische Annalen
   345 (2009), 783–-817.

\bibitem[DT1] {DT1} W. Domitrz, \.{Z}. Tr\c{e}bska, \emph{Symplectic $S_{\mu}$ singularities }, Real and Complex Singularities, Contemporary Mathematics, 569, Amer. Math. Soc., Providence, RI, 2012, 45-65.

\bibitem[DT2] {DT2} W. Domitrz, \.{Z}. Tr\c{e}bska, \emph{Symplectic $T_7, \ T_8$ singularities and Lagrangian tangency orders}, to appear in Proceedings of the Edinburgh Mathematical Society.

\bibitem [G] {G} M. Giusti, \emph{Classification des singularit\'es isol\'ees d'intersections compl\`etes simples},
C. R. Acad. Sci., Paris, Sér. A 284 (1977), 167-170 .

\bibitem [IJ1] {IJ1} G. Ishikawa, S. Janeczko, \emph{ Symplectic bifurcations of plane curves and isotropic liftings},
Q. J. Math. \textbf{54}, No.1 (2003), 73-102.

\bibitem [IJ2] {IJ2} G. Ishikawa, S. Janeczko, \emph{ Symplectic singularities of isotropic mappings},
Geometric singularity theory, Banach Center Publications
\textbf{65} (2004), 85-106.

\bibitem [K] {K} P. A. Kolgushkin, \emph{Classification of simple multigerms of curves
in a space endowed with a symplectic structure}, St. Petersburg
Math. J. \textbf{15} (2004), no. 1, 103-126.

\bibitem[L] {L} E. J. M. Looijenga \emph{Isolated Singular Points on Complete Intersections},
London Mathematical Society Lecture Note Series 77, Cambridge University Press 1984.

\bibitem[T] {T} \.{Z}. Tr\c{e}bska, \emph{Symplectic $W_8$ and $W_9$ singularities}, to appear in Journal of Singularities.

\bibitem [Z] {Zh} M. Zhitomirskii, {\em Relative Darboux
theorem for singular manifolds and local contact algebra}, Can. J.
Math. \textbf{57}, No.6 (2005), 1314-1340.
\end{thebibliography}

\end{document}